\documentclass{amsart}
\usepackage{amsmath}
\usepackage{amssymb}

  \usepackage{paralist}
  \usepackage{graphics} 
  \usepackage{epsfig} 
 \usepackage[colorlinks=true]{hyperref}
\hypersetup{urlcolor=blue, citecolor=red}

\newtheorem{theorem}{Theorem}[section]
\newtheorem{corollary}[theorem]{Corollary}

\newtheorem{lemma}[theorem]{Lemma}
\newtheorem{proposition}[theorem]{Proposition}

\theoremstyle{definition}
\newtheorem{definition}[theorem]{Definition}
\newtheorem*{remark}{Remark}

\newtheorem{example}[theorem]{Example}

\newcommand{\dimb}{\mathrm{dim}_B}
\newcommand{\dimh}{\mathrm{dim}_H}
\newcommand{\calu}{\mathcal{U}}
\newcommand{\diam}{\mathrm{diam}}
\newcommand{\calv}{\mathcal{V}}

\title[Decay of Lebesgue Numbers]
      {Exponential Decay of Lebesgue Numbers}

\author[Peng Sun]{}


 \email{pengsunmath@hotmail.com}

\begin{document}

\maketitle

\centerline{\scshape Peng Sun }
\medskip
{\footnotesize
 \centerline{China Economics and Management Academy}
   \centerline{Central University of Finance and Economics}
   \centerline{No. 39 College South Road, Beijing, 100081, China}
} 

\bigskip

\begin{abstract}
    We study the exponential rate of decay of Lebesgue numbers of open
    covers in topological dynamical systems. We show that topological
    entropy is bounded by this rate multiplied
    by dimension. Some
    corollaries and examples are discussed.
     
\end{abstract}
\section{Motivation}
Entropy, which measures complexity of a dynamical system, has various definitions
in both topological and measure-theoretical contexts. Most of these definitions
are closely related to each other. Given a partition on a measure space,
the famous Shannon-McMillan-Breiman Theorem asserts that for almost every
point the cell covering it, generated under dynamics, decays in measure with
the asymptotic exponential rate equal to the entropy. It is natural to consider
analogous objects in topological dynamics. Instead of measurable partitions,
the classical definition of topological entropy due to Adler, Konheim and
 McAndrew involves open covers, which also generate cells under dynamics.
We would not like to speak of any invariant
measure as in many cases they may be scarce or pathologic, offering us very
little information about the local geometric structure.
Diameters of cells are also useless since usually the image of a cell may spread to the
whole space. Finally we arrive at Lebesgue number. It measures how large
a ball around every point is contained in some cell. It is a global characteristic
but exhibits local facts, in which sense catches some
idea of Shannon-McMillan-Breiman Theorem. We also notice that the results
we obtained provides a good upper estimate of topological entropy which is
computable with reasonable effort.

\section{Preliminaries on Lebesgue number}


First we briefly discuss some preliminaries on Lebesgue number and open covers.
Some of those can be found in any textbook of elementary topology. For
the rest, as well as other facts we discuss in succeeding sections without proof, one can refer to,
for example, \cite{PW}.

The basic result we shall use is the following Lebesgue Covering Lemma.

\begin{theorem}
Let $(X,d)$ be a compact metric space. $\mathcal{U}$ is an open cover of
$X$. Then there is $\delta>0$ such that every open ball of radius at most
$\delta$ is contained in some element of $\mathcal{U}$. We call the largest
such number the Lebesgue number of the open cover.
\end{theorem}

\begin{proof}
If $X\in\mathcal{U}$ then the theorem is trivial.

If $X\notin\mathcal{U}$, let
$$\delta(\mathcal{U},x)=\sup_{U\in\mathcal{U}}(\inf_{y\in X\backslash U}d(x,y)).$$
Then $\delta(\mathcal{U},x)$ is a continuous function on $X$ taking strictly
positive values. Since $X$ is compact, the function attains its minimum value
on $X$. So
$$\delta(\mathcal{U})=\min_{x\in X} \delta(\mathcal{U},x)>0$$
is the Lebesgue number of the open cover.
\end{proof}

\begin{remark}
Another widely used formulation of Lebesgue Covering Lemma states that there
is $\bar\delta>0$ (the largest one) such that every set of diameter less
than $\bar\delta$
is contained in some element of $\mathcal{U}$. It is easy to see $\delta\le\bar\delta\le
2\delta$. This guarantees that Definition \ref{Lentropy} is not affected
if Lebesgue number is defined this way.
\end{remark}

We have some simple facts on Lebesgue numbers.

\begin{lemma}\label{finer}
For two open covers $\mathcal{U}$ and $\mathcal{V}$, we say $\mathcal{U}$
is finer than $\mathcal{V}$, denoted by $\calu\succ\calv$, if every element of $\mathcal{U}$ is contained
in some element of $\mathcal{V}$. If $\mathcal{U}$ is finer than $\mathcal{V}$
 then $\delta(\mathcal{U})\le\delta(\mathcal{V})$.
\end{lemma}

\begin{lemma}\label{shk}
Let $\mathrm{diam}(\mathcal{U})=\sup_{U\in\mathcal{U}}\mathrm{diam}(U)$.
For every open covers $\calu$ and $\calv$, if $\diam(\calu)<\delta(\calv)$,
then $\calu\succ\calv$. 
\end{lemma}

\begin{lemma}
For any two open covers $\mathcal{U}$ and $\mathcal{V}$, let
$$\mathcal{U}\bigvee\mathcal{V}=\{U\cap
V|U\in\mathcal{U},V\in\mathcal{V}\}.$$ Then
$$\delta(\mathcal{U}\bigvee\mathcal{V})=\min\{\delta(\mathcal{U}),
\delta(\mathcal{V})\}.$$
\end{lemma}

\begin{proof}
On one hand, $\mathcal{U}\bigvee\mathcal{V}$ is finer than $\mathcal{U}$
and $\mathcal{V}$. By Proposition \ref{finer}
$$\delta(\mathcal{U}\bigvee\mathcal{V})\le\min\{\delta(\mathcal{U}),
\delta(\mathcal{V})\}$$
On the other hand, for every $x$, there are $U_x\in\mathcal{U}$ and $V_x\in\mathcal{V}$
such that
$$\delta(\mathcal{U},x)=\min_{y\in X\backslash U_x} d(x,y)\ge\delta(\mathcal{U}),
\;\delta(\mathcal{V},x)=\min_{y\in X\backslash V_x} d(x,y)\ge\delta(\mathcal{V})$$
Then $$\delta(\mathcal{U}\bigvee\mathcal{V},x)\ge\min_{y\in X\backslash (U_x\cap
V_x)} d(x,y)=\min\{\delta(\mathcal{U},x),\delta(\mathcal{V},x)\}
\ge\min\{\delta(\mathcal{U}),
\delta(\mathcal{V})\}$$
\end{proof}

Now let $f$ be a continuous map from $X$ to itself.
Let
$$\mathcal{U}_f^n=\bigvee_{k=0}^{n-1} f^{-k}(\mathcal{U}),\;\delta_n=\delta_n(f, \mathcal{U})=\delta(\mathcal{U}_f^n)$$
where $f(\mathcal{U})=\{f(U)|U\in\mathcal{U}\}$.

\begin{corollary}\label{mindel}
Let $\mathcal{U}$ be an open cover, then
$$\delta_n(f,\mathcal{U})=\min_{0\le k\le n-1}\delta(f^{-k}(\mathcal{U})).$$
\end{corollary}

\begin{corollary}\label{finless}
If $\calu\succ\calv$, then for every $n$, we have $f^{-n}(\calu)\succ f^{-n}(\calv)$
and $\calu_f^n\succ\calv_f^n$, hence $\delta(f^{-n}(\calu))\le \delta(f^{-n}(\calv))$
and $\delta_n(f,\calu)\le\delta_n(f,\calv)$.
\end{corollary}

\section{Decay of Lebesgue numbers}

Now we turn to the asymptotic decay of Lebesgue numbers.

\begin{definition}\label{Lentropy}
Let $\mathcal{U}$ be an open cover of $X$. We set
$$h_L^-(f,\mathcal{U})=\liminf_{n\to\infty}-\frac1n\log\delta_n(f,\mathcal{U}),$$
$$h_L^+(f,\mathcal{U})=\limsup_{n\to\infty}-\frac1n\log\delta_n(f,\mathcal{U}),$$
$$h_L^-(f)=\sup h_L^-(f,\mathcal{U})$$
and
$$h_L^+(f)=\sup h_L^+(f,\mathcal{U}).$$
Here the supremums are taken over all finite open covers.

\end{definition}

From now on we use $h_L^*$ to denote either $h_L^+$ or $h_L^-$, when the argument works for both cases. We note that these numbers possess some properties
analogous to entropy.

\begin{lemma}
For every continuous map $f$ and every open cover $\calu$, we have:
\begin{enumerate}
\item $h_L^*(f,\calu)\ge 0$, hence $h_L^*(f)\ge 0$.
\item If $f$ is an isometry, then $h_L^*(f)=h_L^*(f,\calu)=0$.
\item $h_L^-(f,\calu)\le h_L^+(f,\calu)$, hence $h_L^-(f)\le h_L^+(f)$.
\item For every $n>0$, $h_L^*(f,\calu)=h_L^*(f,f^{-n}(\calu))=h_L^*(f,\calu_f^n)$.
If in addition $f$ is a homeomorphism, then the first equality also holds
for $n<0$.
\item If $\calu\succ\calv$, then $h_L^*(f,\calu)\ge h_L^*(f,\calv)$.
\end{enumerate}
\end{lemma}

\begin{proposition}\label{multi}
For every $n>0$, $h_L^*(f^n)=n h_L^*(f)$.
\end{proposition}

\begin{proof}
On one hand, for every finite open cover $\mathcal{U}$ and $m>0$, by Corollary
\ref{mindel}
$$\delta_m(f^n,\mathcal{U})=\min_{0\le j\le m-1}\delta(f^{-jn}(\mathcal{U}))\ge\min_{0\le
j\le mn-1}\delta(f^{-j}(\mathcal{U}))=\delta_{mn}(f,\mathcal{U}).$$
Taking limit we obtain $h_L^*(f^n,\mathcal{U})\le n h_L^*(f,\mathcal{U})$,
hence $h_L^*(f^n)\le n h_L^*(f)$.

On the other hand,
$$\delta_{m}(f^n,\mathcal{U}_f^n)=\min_{0\le
j\le m-1}\delta(f^{-jn}(\mathcal{U}_f^n))=\min_{0\le
j\le mn-1}\delta(f^{-j}(\mathcal{U}))=\delta_{mn}(f,\mathcal{U}).$$
This implies $h_L^*(f^n,\mathcal{U}_f^n)\ge n h_L^*(f,\mathcal{U})$.
So $h_L^*(f^n)\ge n h_L^*(f)$.
\end{proof}

Applying Corollary \ref{mindel}, we also have:

\begin{corollary}\label{nine}
$$h_L^+(f,\calu)\ge\limsup_{n\to\infty}-\frac1n\log\delta(f^{-n}(\calu)).$$
$$h_L^-(f,\calu)\ge\liminf_{n\to\infty}-\frac1n\log\delta(f^{-n}(\calu)).$$
(We intentionally replace $f^{1-n}$ by $f^{-n}$ in the limits.)
\end{corollary}

In fact, we have:

\begin{proposition}\label{eqdef}
$$h_L^+(f,\calu)=\limsup_{n\to\infty}-\frac1n\log\delta(f^{-n}(\calu))$$
\end{proposition}

\begin{remark}
The analogous result is not necessarily true for $h_L^-$.
\end{remark}

The proposition is a corollary of the following lemma.
We also note that Lebesgue
number is always bounded by the diameter of the space.

\begin{lemma}
Let $a_n\ge K$ be real numbers, uniformly bounded from below. Let
$$b_n=\max\{a_k|1\le k\le n\}.$$
Then
$$\limsup_{n\to\infty}\frac{ a_n}n=\limsup_{n\to\infty}\frac{ b_n}n$$
\end{lemma}

\begin{proof}
By definition, $a_n\le b_n$. So $$\limsup_{n\to\infty}\frac{ a_n}n\le\limsup_{n\to\infty}\frac{ b_n}n.$$

Note that $\{b_n\}$ is a non-decreasing sequence. If there is $N$ such that
for all $n>N$, $b_n=b_N$, then
$$\limsup_{n\to\infty}\frac{ b_n}n=0=\lim_{n\to\infty}\frac Kn\le\limsup_{n\to\infty}\frac{ a_n}n.$$

Otherwise, the set $J=\{n_j|b_{n_j-1}<b_{n_j}\}$ has infinitely many elements.
If $n_j\in J$ then we must have $a_{n_j}=b_{n_j}$.
$$\limsup_{n\to\infty}\frac{ b_n}n=\limsup_{j\to\infty}\frac{ b_{n_j}}{n_j}=\limsup_{j\to\infty}\frac{ a_{n_j}}{n_j}\le\limsup_{n\to\infty}\frac{ a_n}n.$$

\end{proof}

\begin{proposition}\label{eql}
$$h_L^*(f)=\lim_{\epsilon\to 0}\inf_{\mathrm{diam}(\mathcal{U})<\epsilon}h_L^*(f,\mathcal{U}).$$
\end{proposition}

\begin{proof}

By definition, $h_L^*(f)\ge h_L^*(f,\calu)$ for every open cover $\calu$.
So
$$h_L^*(f)\ge\limsup_{\epsilon\to 0}\inf_{\mathrm{diam}(\mathcal{U})<\epsilon}h_L^*(f,\mathcal{U}).$$

For every $\theta>0$, if $h_L^*(f,\calv)>h_L^*(f)-\theta$, then
for every $\epsilon<\delta(\calv)$ and every open cover $\calu$ with
$\diam(\calu)<\epsilon$, we have $h_L^*(f,\calu)\ge h_L^*(f,\calv)>h_f^*(f)-\theta$.
So $$h_L^*(f)\le\liminf_{\epsilon\to 0}\inf_{\mathrm{diam}(\mathcal{U})<\epsilon}h_L^*(f,\mathcal{U}).$$

\end{proof}

\begin{remark}
This proposition implies that in Definition \ref{Lentropy} the supremums
may be taken over all open covers (not necessarily finite).
\end{remark}

\begin{corollary}
For every sequence of open covers $\{\calu_k\}_{k\ge
1}$,
if $\diam(\calu_k)\to 0$, then $h_L^*(f,\calu_k)\to h_L^*(f)$.
\end{corollary}

\begin{corollary}
If $f$ is an expansive homeomorphism with expansive constant $\gamma$,
then for every open cover $\calu$ of diameter less than $\gamma$,
$h_L^*(f,\calu)=h_L^*(f)$.
\end{corollary}

\begin{proof}
By assumption, every element of 
$\bigvee_{n=-\infty}^\infty f^n(\overline\calu)$
contains at most one point. $\calu$ is a generator.
By \cite[Theorem 5.21]{PW}, for every $\epsilon>0$ there is $N>0$ such that
$\diam(\bigvee_{n=-N}^N f^n(\calu))<\epsilon$. But $h_L^*(f,\calu)=h_L^*(f,\bigvee_{n=-N}^N f^n(\calu))$. So $h_L^*(f,\calu)=h_L^*(f)$.
\end{proof}

\begin{proposition}\label{lbd}
Let $X_\infty=\bigcap_{n=0}^\infty f^{n}(X)$.
If $x,y\in X_\infty$, for each $n>0$, let
$$D(f^{-n}(x),f^{-n}(y))=\inf\{d(z,z'):f^n(z)=x, f^n(z')=y\}.$$
If $X_\infty$ is not a single point, then
\begin{align*}
h_L^-(f)\ge&\sup_{x\ne y\in X_\infty}\liminf_{n\to\infty}\frac1n\log\frac{d(x,y)}{D(f^{-n}(x),f^{-n}(y))},\\
h_L^+(f)\ge&\sup_{x\ne y\in X_\infty}\limsup_{n\to\infty}\frac1n\log\frac{d(x,y)}{D(f^{-n}(x),f^{-n}(y))}.
\end{align*}
\end{proposition}

\begin{proof}
We only show the first inequality. Proof of the other one is similar.

Let $$\sup_{x,y\in X}\liminf_{n\to\infty}\frac1n\log\frac{d(x,y)}{d(f^{-n}(x),f^{-n}(y))}=\lambda$$
Then for every $\epsilon>0$, there are $x_0,y_0\in X_\infty$ such that
$$\liminf_{n\to\infty}\frac1n\log\frac{d(x_0,y_0)}{d(f^{-n}(x_0),f^{-n}(y_0))}>\lambda-\epsilon$$
So there is $n_0>0$ such that if $n>n_0$ then
$$\frac1n\log\frac{d(x_0,y_0)}{d(f^{-n}(x_0),f^{-n}(y_0))}>\lambda-2\epsilon$$
For any open cover $\mathcal{U}$ of diameter less than $d(x_0,y_0)$,
$y_0$ is not covered by any element of $\mathcal{U}$ covering $x_0$.
So every element of $f^{-n}(\calu)$ covering a point of $f^{-n}(x_0)$ can
not cover any point of $f^{-n}(y_0)$. 
This implies that $$\delta_{n+1}(f,\mathcal{U})<d(f^{-n}(x_0),f^{-n}(y_0))$$
$$h_L^-(f,\mathcal{U})=\liminf_{n\to\infty}-\frac1n\log \delta_n(f,\mathcal{U})\ge\lambda-2\epsilon$$
Apply Proposition \ref{eql} and let $\epsilon\to 0$, then we obtain $h_L^-(f)\ge\lambda$.
\end{proof}

\begin{remark}
The inequalities may be strict. See Example \ref{exlbd}.
\end{remark}



%

\section{Lebesgue numbers, entropy and dimensions}

In this section we investigate the relations between decay of Lebesgue numbers,
topological entropy and sorts of dimensions. We consider the three definitions
of topological entropy: one using open covers, oene using separated sets,
and one closely related to Hausdorff dimension. Each of them has something
to do with Lebesgue numbers.

\subsection{Lebesgue numbers and minimal covers}
Denote by $S(\mathcal{U})$ the smallest cardinality of a sub-cover of $\mathcal{U}$.
For a given continuous map $f$
on a compact metric space $X$, the topological entropy of $\mathcal{U}$
is
$$h(f,\mathcal{U})=\lim_{n\to\infty}\frac1n\log S(\mathcal{U}_f^n).$$
The topological entropy $h(f)$ of $f$ is then defined as the maximum of $h(f,\mathcal{U})$
taking
over all finite open covers of $X$.

 Denote by $B(x,\gamma)=\{y\in X| d(x,y)<\gamma\}$ the open $\gamma$-ball
around $x$. Let $N(\gamma)$ be the minimal number of $\gamma$-balls needed
to cover $X$. The upper box dimension of $X$ is defined by
$$\dimb(X)=\limsup_{\gamma\to 0}-\frac{\log N(\gamma)}{\log \gamma}.$$
It is clear that if $\gamma_1\ge\gamma_2$ then $N(\gamma_1)\le N(\gamma_2)$.

\begin{lemma}
For every open cover $\mathcal{U}$, if $\mathcal{W}$ is a minimal
sub-cover then
$N(\delta(\mathcal{W}))\ge |\mathcal{W}|=S(\mathcal{U})$.

\end{lemma}

\begin{proof}
Let $\mathcal{W}=\{W_j\}$. Then for each $j$, there is $x_j\in W_j$ such
that $x_j\notin W_k$ for all $k\neq j$. Otherwise $\mathcal{W}-\{W_j\}$ is
a sub-cover of $\mathcal{U}$ with smaller cardinality.

As every $\delta(\mathcal{W})$-ball is covered by some element of $\mathcal{W}$,
it can cover at most
one element in $\{x_j\}_{1\le j\le |\mathcal{W}|}$.
So the minimal number of $\delta(\mathcal{W})$-balls needed to cover
$X$ is no less than $|\mathcal{W}|$.
\end{proof}

\begin{corollary}
Let
$\Delta(\mathcal{U})=\max\{\delta(\mathcal{W})|\mathcal{W}\text{ is a
minimal sub-cover of }\mathcal{U}\}$.
Then $N(\Delta(\mathcal{U}))\ge S(\mathcal{U})$.
\end{corollary}

\begin{definition}
For every open cover $\mathcal{U}$,
let $\Delta_n=\Delta(\mathcal{U}_f^n)$. We set
$$h_L^\Delta(f,\calu)=\liminf_{n\to\infty}-\frac1n\log\Delta_n$$
and
$$h_L^\Delta(f)=\sup_{\mathcal{U}}h_L^\Delta(f,\calu),$$
where the supremum is taken over all finite open covers.
\end{definition}

\begin{remark}
For every $\calu$, we have $\delta(\calu)\ge\Delta(\calu)$. So
$h_L^\Delta(f,\calu)\ge h_L^-(f,\calu)$ and $h_L^\Delta(f)\ge h_L^-(f)$.
\end{remark}

\begin{theorem}\label{entpar}
$$\dimb(X)\cdot h_L^\Delta(f,\calu)\ge h(f,\mathcal{U}).$$
\end{theorem}

\begin{proof}

If $h(f,\mathcal{U})=0$ then the theorem is trivial since $\Delta_n$ is non-increasing.

If $h(f,\mathcal{U})>0$ then as $n\to\infty$, $S(\mathcal{U}_f^n)\to\infty$.
But $N(\Delta_n)\ge S(\mathcal{U}_f^n)$, we must have $\Delta_n\to 0$.

Fix a small $\epsilon>0$. For every $N_0>0$, by definition of the upper box dimension,
there is $\gamma_0>0$ and $N_1>N_0$ such that $\Delta_n<\gamma_0$ for all $n>N_1$, and
$$-\frac{\log N(\Delta_n)}{\log \Delta_n}<\dimb(X)+\epsilon.$$
Hence
$$-\log\Delta_n\cdot({\dimb(X)+\epsilon})>\log N(\Delta_n)\ge\log S(\mathcal{U}_f^n),$$
$$\liminf_{n\to\infty}-\frac1n\log\Delta_n
\cdot{\dimb(X)+\epsilon}\ge \lim_{n\to\infty}\frac1n\log S(U_f^n).$$
Let $\epsilon\to 0$ then the result follows.
\end{proof}

\begin{corollary}
$$\dimb(X)\cdot h_L^\Delta(f)\ge h(f).$$
\end{corollary}


\subsection{Lebesgue numbers and separated sets}
The last subsection is just a warm-up. Usually, it is not convenient
to refer to Lebesgue number of the minimal cover as it might be quite different
from Lebesgue number of the original one. So we shall not focus on $h_L^\Delta$
but turn to $h_L^-$ and $h_L^+$.

For a continuous map $f$ on a compact metric space $(X,d)$, we define for
each $n$ a metric
$$d_f^n(x,y)=\max_{0\le k\le n-1} d(f^k(x),f^k(y)).$$
Recall for $\epsilon>0$, $E\subset X$ is said to be an $(n,\epsilon)$-separated set if $d_f^n(x,y)>\epsilon$ for distinct points $x,y\in E$. Let
$s_n(f,\epsilon)=\max |E|$, where the maximum is taken over all $(n,\epsilon)$-separated
sets. Then 
$$h(f)=\lim_{\epsilon\to 0}\limsup_{n\to\infty}\frac1n \log s_n(f,\epsilon)=\lim_{\epsilon\to 0}\liminf_{n\to\infty}\frac1n \log s_n(f,\epsilon).$$

Now we consider an open cover $\mathcal{U}=\{U_\alpha\}_{\alpha\in I}$ of
$X$. 

\begin{lemma}
For a given open cover $\mathcal{U}$ and $\epsilon>0$, if $\diam(\mathcal{U})<\epsilon$,
then for each $n$, $N(\delta_n(f,\mathcal{U}))>s_n(f,\epsilon)$.
\end{lemma}

\begin{proof}

Let $E$ be an $(n,\epsilon)$-separated set of cardinality $s_n(f,\epsilon)$.
If distinct points $x,y\in E$ are covered by the same $\delta_n$-ball, then
the ball is  covered by some element
$$V=\bigcap_{k=0}^{n-1} f^{-k}(U_{i_k})\subset \mathcal{U}_f^n,$$
where $U_{i_k}\in\mathcal{U}$ is some element of
$\mathcal{U}$ for each $k$.
This implies that $x,y\in V$ and $f^k(x),f^k(y)\in
U_{i_k}$. Since $\diam(\mathcal{U})<\epsilon$, $d(f^k(x),f^k(y))<\epsilon$ for
$0\le k\le n-1$. So $d_f^n(x,y)<\epsilon$, which contradicts the fact that
$x$ and $y$ are $(n,\epsilon)$-separated.

So each $\delta_n$-ball can cover at most one point in $E$. $N(\delta_n)>s_n(f,\epsilon)$.
\end{proof}

\begin{theorem}\label{lebent}
$$\dimb(X)\cdot h_L^-(f)\ge h(f)$$
\end{theorem}

\begin{proof}
If $h(f)=0$ then it is trivial.

If $h(f)>0$ then for all small $\epsilon>0$, as $n\to\infty$, $s_n(f,\epsilon)\to\infty$.
But for every open cover $\mathcal{U}$ such that $\diam(\mathcal{U})<\epsilon$,
we have $N(\delta_n)\ge s_n(f,\epsilon)$. Hence $\delta_n\to 0$.

Fix a small $\theta>0$. For every $N_0>0$, by definition of the upper box dimension,
there is $\gamma_0>0$ and $N_1>N_0$ such that $\delta_n<\gamma_0$ for all $n>N_1$, and
$$-\frac{\log N(\delta_n)}{\log \delta_n}<\dimb(X)+\theta$$
Hence
$$-\log\delta_n\cdot({\dimb(X)+\theta})>\log N(\delta_n)\ge \log s_n(f,\epsilon)$$
$$(\dimb(X)+\theta)\cdot\liminf_{n\to\infty}-\frac1n\log\delta_n(f,\mathcal{U})\ge\lim_{n\to\infty}\frac1n \log s_n(f,\epsilon)$$
This is true for every $\mathcal{U}$ with diameter less than $\epsilon$ and
every $\theta>0$.
Let $\epsilon\to 0$ and apply Proposition \ref{eql}.
\end{proof}


\subsection{Hausdorff dimension and Bowen's definition of topological entropy}
This part has been inspired by \cite{Misiu}. Instead of box dimension we
now discuss Hausdorff dimension. 
By considering Lebesgue numbers we obtain an inequality relating Hausdorff
dimension and topological entropy, which implies \cite[Theorem
2.1]{Misiu} and \cite[Theorem 2]{DZG} (See Corollary \ref{Lip}).

Recall Bowen's definition of topological entropy \cite{Bowen} that is equivalent
to those we have discussed as the space $X$ is assumed to be compact.
Let $\mathcal{U}$ be a finite open cover of $X$. For a set $B\subset X$ we
 write $B\succ\mathcal{U}$ if $B$ is
contained in some element of $\mathcal{U}$.
Let $n_{f,\mathcal{U}}(B)$ be the largest nonnegative integer $n$
such that $f^k(B)\succ \mathcal{U}$ for $k = 0, 1,\dots,n-1$. If $B\nprec\mathcal{U}$
then $n_{f,\mathcal{U}}(B) = 0$ and if
$f^k(B)\succ\mathcal{U}$ for all $k$ then $n_{f,\mathcal{U}}(B)=\infty$.
Now we set $\mathrm{diam}_\mathcal{U}(B) = \exp(-n_{f,\mathcal{U}}(B))$.
If $\mathcal{B}$ is also a cover of $X$, we set
$$\mathrm{diam}_\mathcal{U}(\mathcal{B}) = \sup_{B\in\mathcal{B}}
\mathrm{diam}\mathcal{U}(B)$$
and for any real number $\lambda$,
$$D_{\mathcal{U}}(\mathcal{B}, \lambda)=\sum_{B\in\mathcal{B}}
(\mathrm{diam}_\mathcal{U}(B))^\lambda.$$
Then there is a number $h_\mathcal{U}(f)$ such that
$$\mu_{\mathcal{U},\lambda}(X)=\lim_{\epsilon\to 0}
\inf\{D_\mathcal{U}(\mathcal{B}, \lambda)| \mathcal{B}\text{ is a cover of
} X\text{ and }\mathrm{diam}_\mathcal{U}(\mathcal{B}) < \epsilon\}$$
is $\infty$ for $\lambda<h_\mathcal{U}(f)$ and $0$ for $\lambda>h_\mathcal{U}(f)$.
As showed in \cite{Bowen}, we have
$$h(f)=\sup\{h_\mathcal{U}(f)| \mathcal{U}\text{ is a finite open cover of
} X\}.$$

The classical Hausdorff measure is defined as
$$\mu_\lambda(X)=\lim_{\epsilon\to 0}\inf\{\sum_{U\in\mathcal{U}}(\diam(U))^\lambda|
\mathcal{U} \text{ is a cover of } X \text{ and } \diam(\mathcal{U})<\epsilon\}.$$
We know the Hausdorff dimension of $X$ is a number $\dimh(X)$ such that $\mu_\lambda(X)=\infty$
for $\lambda<\dimh(X)$ and $\mu_\lambda(X)=0$ for $\lambda>\dimh(X)$.

\begin{theorem}
$$\dimh(X)\cdot h_L^+(f,\mathcal{U})\ge h_\mathcal{U}(f).$$
\end{theorem}

\begin{proof}
Take $K>h_L^+(f,\mathcal{U})\ge 0$. Then there is $n_0$ such that for every $n>n_0$,
$$-\frac1{n-1}\log\delta_n(f,\mathcal{U})<K.$$
For $B\in X$, if 
\begin{equation}\label{dia}
\diam(B)\le\exp -K(n-1)<\delta_n(f,\mathcal{U})
\end{equation}
then $n_{f,\mathcal{U}}(B)\ge n$ since $B$ is contained in $\mathcal{U}_f^n$.
(\ref{dia}) is satisfied for
$$n\le-\log(\diam(B))/K+1.$$
So
\begin{equation}\label{nfa}
n_{f,\mathcal{U}}(B)>-\log(\diam(B))/K\;\text{and}\;\mathrm{diam}_\mathcal{U}(B)<(\diam(B))^{1/K}.
\end{equation}

Now fix $\lambda>\dimh(X)$. By the definition of Hausdorff dimension, $\mu_\lambda(X)=0$.
For every
$\epsilon>0$ small enough (much smaller than $\delta_{n_0}(f,\mathcal{U})$) and
every small $\gamma>0$, there is a cover $\mathcal{B}$ such that $\diam(\mathcal{B})<\epsilon$
and
$$\sum_{B\in\mathcal{B}}(\diam(B))^\lambda<\gamma.$$
By (\ref{nfa}) we have
$$D_{\mathcal{U}}(\mathcal{B}, \lambda K)=\sum_{B\in\mathcal{B}}
(\mathrm{diam}_\mathcal{U}(B))^{\lambda K}<\sum_{B\in\mathcal{B}}(\diam(B))^\lambda<\gamma$$
while $\mathrm{diam}_\mathcal{U}(\mathcal{B})<\epsilon^{1/K}$.
This implies that $\mu_{\mathcal{U},\lambda K}(X)=0$.
Hence $\lambda K>h_\mathcal{U}(f)$, whenever $\lambda>\dimh(X)$
and $K>h_L^+(\mathcal{U})$. So $\dimh(X)\cdot h_L^+(\mathcal{U})\ge h_\mathcal{U}(f)$.
\end{proof}

\begin{remark}
Similarly, for $Y\subset X$ we can define $\mu_{\mathcal{U},\lambda}(Y)$,
$h_\mathcal{U}(f,Y)$,
$\mu_{\lambda}(Y)$ and $\dimh(Y)$ (as in \cite{Misiu}). It is not difficult
to show that if $\mathcal{U}$ is a finite open cover of $Y$, then
$$\dimh(Y)\cdot h_L^+(f,\mathcal{U})\ge h_\mathcal{U}(f,Y).$$
\end{remark}


\begin{corollary}\label{hausleb}
$$\dimh(X)\cdot h_L^+(f)\ge h(f).$$
\end{corollary}

\section{Lipschitz maps}
We have shown that $h_L^-$ and $h_L^+$ provide upper estimates of topological
entropy.
Now we show that these numbers are bounded for Lipschitz
maps. Recall that a continous map $f$ is Lipschitz with constant $L(f)>0$
if for every $x,y\in X$, $d(f(x),f(y))\le L(f)\cdot d(x,y)$.
Here we assume that $L(f)$ to be the smallest one among such numbers.

\begin{theorem}\label{LLL}
If $f$ is Lipschitz with constant $L(f)$, then for every finite open cover
$\mathcal{U}$, $h_L^+(f,\mathcal{U})\le\max\{\log L(f),0\}$.
\end{theorem}

\begin{proof}
Let $L=\max\{L(f),0\}$.
For every $x\in X$ and every $y\in B(x,\delta(\mathcal{U})\cdot L^{-(n-1)})$,
$$d(f^j(x),f^j(y))\le L^j\cdot d(x,y)\le \delta(\mathcal{U})$$
for $j=0,1,\dots,n-1$.
This implies that
$$f^j(B(x,\delta(\mathcal{U})\cdot L^{-(n-1)}))\subset B(f^j(x),\delta(\mathcal{U}))\subset
U_j$$
for some $U_j\in\mathcal{U}$, $j=0,1,\dots,n-1$. So $\delta(\mathcal{U}_f^n,x)\ge\delta(\mathcal{U})\cdot
L^{-(n-1)}$ for every $x\in X$, hence $\delta_n\ge\delta(\mathcal{U})\cdot L^{-(n-1)}$.
$$h_L^+(f,\mathcal{U})=\limsup_{n\to\infty}-\frac1n\log\delta_n\le\limsup_{n\to\infty}-\frac1n\log(\delta(\mathcal{U})\cdot
L^{-(n-1)})=\log L.$$
\end{proof}

\begin{corollary}\label{Lip}\textnormal{(see also \cite{DZG}\cite{Misiu})}
If $f$ is Lipschitz with constant $L(f)>1$, then for every $Y\subset X$,
$$\dimh(Y)\cdot\log L(f)\ge {h(f,Y)}.$$
In particular,
$$\dimh(X)\cdot\log L(f)\ge h(f).$$
\end{corollary}

\begin{remark}
We note (thanks to Anatole Katok) long before Bowen's definition of topological entropy was introduced, the weaker result involving the box dimension (see, e.g.\cite[Theorem
3.2.9]{KB}) had been proved by Kushnirenko:
$$\dimb(X)\cdot\max\{\log L(f),0\}\ge h(f).$$
\end{remark}

\begin{corollary}\label{hauslip}
If $f$ is Lipschitz, let $$l(f)=\inf\{\frac1n\log L(f^n)|{n\ge 1}\}.$$
Then for every finite open cover $\mathcal{U}$, $h_L^+(f,\mathcal{U})\le\max\{l(f),0\}$.

If $l(f)>0$, then for every $Y\subset X$,
$$\dimh(Y)\cdot l(f)\ge h(f,Y).$$
\end{corollary}

\begin{remark}
If $f$ is Lipschitz ($L(f)<\infty$), then the sequence $\{\log L(f^n)\}_{1\le\
n\le\infty}$ is sub-additive. In this case
$$l(f)=\lim_{n\to\infty}\frac1n\log L(f^n).$$

In general, $h_L^-(f)$, $h_L^+(f)$ and $l(f)$ may be different from each
other. Some examples will be discussed in the next section.
\end{remark}

\begin{theorem}
$h_L^*(f)$ is invariant under bi-Lipschitz conjugacy.
\end{theorem}

\begin{proof}
Let $H$ be a bi-Lipschitz conjugacy between $f$ on $X$ and $g$ on $Y$.
For a finite open cover $\mathcal{U}$ of $X$ and every $x\in X$, there is
$U\in\mathcal{U}$ such that 
\begin{equation*}
U\supset B(x,\delta(\mathcal{U}))\supset H^{-1}(B(H(x),\delta(\mathcal{U})\cdot
L(H^{-1})^{-1}).
\end{equation*}
Then
$$B(H(x), \delta(\mathcal{U})\cdot L(H^{-1})^{-1})\subset H(U)\in H(\mathcal{U}).$$
As $H$ is a homeomorphism, this implies 
\begin{equation}\label{conj}
\delta(H(\mathcal{U}))\ge\delta(\mathcal{U})\cdot L(H^{-1})^{-1}.
\end{equation}
Moreover, $H$ is a conjugacy, $$g^{-n}(H(\mathcal{U}))=H(f^{-n}(\mathcal{U})).$$
Replace $\mathcal{U}$ by $g^{-n}(H(\mathcal{U}))$ in (\ref{conj}), then 
$$\delta(g^{-n}(H(\mathcal{U})))\ge\delta(H(f^{-n}(\mathcal{U})))\cdot
L(H^{-1})^{-1}$$
and hence
$$\delta_n(g,H(\mathcal{U}))\ge \delta_n(f,\mathcal{U})\cdot L(H^{-1})^{-1}.$$
Taking the upper limit we have
$h_L^*(f)\ge h_L^*(g)$. The other direction is the same.
\end{proof}

\begin{remark}
$h_L^*$ depends on the metric chosen and is
not a topological invariant. By the above theorem each of them is the same
for strong equivalent metrics ($C_1d(x,y)<d'(x,y)<C_2d(x,y)$). 
Box dimension and Hausdorff dimension also depend on the metric. However,
the inequalities we obtained holds for any metric, making entropy, a topological
invariant, bounded by geometric numbers.
\end{remark}

\section{Examples}

To finish this paper, we put here several examples.

\begin{example}
Let $f:[0,1]\to[0,1]$ be defined by $f(x)=\sqrt x$. Then $h_L^*(f)=\infty$.

\begin{proof}
Take any finite open cover $\mathcal{U}$ with diameter less than $1/10$.
Then $1/2\notin U$
for every $U\in\mathcal{U}$ covering $0$.
$$\delta_n(f,\mathcal{U})\le\delta(f^{-n}(\mathcal{U}))\le |f^{-n}(1/2)-f^{-n}(0)|=2^{-2^n}.$$
So $h_L^*(f)=\infty$.
\end{proof}

This example 
shows that the numbers $h_L^*(f)$ may be unbounded even if $h(f)=0$. It also shows that $h_L^*(f^{-1})$ may be different from $h_L^*(f)$,
when $f$ is a homeomorphism (here $h_L^*(f^{-1})=0$).

\end{example}

\begin{example}
Let $f:[0,1]\to[0,1]$ be defined by $f(x)=\sqrt{1-x^2}$, then $f$ is not
Lipschitz. But as $f^2(x)=x$, we have
$h_L^*(f,\calu)=0$ for every open cover $\calu$. So $h_L^*(f)=0$.
\end{example}

The following examples show that even for a Lipschitz map $f$, in the inequality
$h_L^-(f)\le h_L^+(f)\le l(f)$, every relation may be strict. There
can be orbits with expanding rates $l(f)$
of arbitrary but finite length. So for any fixed open cover the decay will
no longer depend on $l(f)$ after finite iteration. These examples illustrate
this mechanism and in fact can be modified to be homeomorphisms.

\begin{example}\label{ex3}
Let us consider the compact space $$X=\{0\}\cup\{2^{-m}:m=1,2,\dots\}$$ with
the induced metric and topology from $\mathbb{R}$.
Define $f$ on $X$ by
$$f(x)=\begin{cases}x,&x=2^{-2^k}\text{ for some integer }k \text{ or } x=0;\\
2x, & \text{otherwise}.\end{cases}$$
It is easy to check that $f$ is continuous.

Clearly $L(f)=2$. For every large $n$ we have $f^n(2^{-2^{n+1}+1})=2^{-2^{n+1}+n+1}$ so that
$L(f^n)\ge 2^n$. So $l(f)=\log 2$.

On the other hand, denote by $$X_N=\{0\}\cup\{2^{-m}:m>N\}.$$
If $\mathcal{U}$ is an open cover of $X$, then there is
an element $U_0$ of $\mathcal{U}$ and $N_0>0$ such that $U_0\supset X_{N_0}$.
Let $k_0$ be an integer such that $2^{k_0-1}\le N_0<2^{k_0}$. Then for every
$n$, $f^{-n}(U_0)\supset X_{2^{k_0}}$. As $f^{-n}(\mathcal{U})$ is still
an open cover, we have $$\delta(f^{-n}(\mathcal{U}))\ge d(2^{-2^{k_0}}, 2^{-2^{k_0}+1}),$$
which is independent of $n$. So by Proposition \ref{eqdef}, $h_L^+(f)=0$.
\end{example}

\begin{example}\label{diffl}
Fix $a\ge b>1$. Consider the compact space
\begin{align*}
X_{a,b}=&\{0\}\cup \{a^{-m}:m=1,2,\dots\}\\
&\cup\{\frac{qa^{-2^p}}{a+q}:p\in\mathbb{N}, q\in\mathbb{N}\}\\
&\cup\{a^{-2^p}(1+ b^{q}):p\in\mathbb{N},
q\in\mathbb{Z}, 1+b^q<a\}
\end{align*}
with the induced metric and topology from $\mathbb{R}$. Define $f$ on $X$ by
$$f(x)=\begin{cases}
x,& x=2^{-2^k}\text{ for some integer }k \text{ or } x=0;\\
\min\{y\in X_{a,b}: y>x\}, & \text{otherwise}.
\end{cases}$$

It is easy to check that $f$ is a homeomorphism. Similar argument as Example
\ref{ex3} shows that $l(f)=\log a$ and $h_L^*(f)=\log b$.

This example also shows that $h_L^*(f)$ is far from a topological invariant
since all these functions are topologically conjugate for arbitrary values
of $a$ and $b$.
\end{example}

\begin{example}
In Example \ref{diffl} we can replace $X_{a,b}$ by
\begin{align*}
X_{a}=\{0\}\cup \{a^{-m}:m=1,2,\dots\}\cup\{a^{-2^p}(\frac{q}{a+q})^{\pm 1}:p\in\mathbb{N}, q\in\mathbb{N}\}.
\end{align*}
Then $l(f)=\log a$ and $h_L^*(f)=0$. This together with Example \ref{diffl}
imply that for every $l\ge 0$, the
collection of possible values
$$\{h_L^*(f): f \text{ is a homeomorphism and } l(f)=l\}$$
is the whole interval $[0,l]$.

Moreover, if we consider $g=(f,Id)$ on $X_{a}\times[0,1]$, then
we have $\dimh(X)=1, h(g)=0, l(g)=\log a$ and $h_L^*(g)=0$.
This shows that Corollary \ref{hausleb} can be a strictly better estimate than
Corollary \ref{hauslip} (also the results in \cite{DZG}\cite{Misiu}).
\end{example}

\begin{example}
Let $a>b>0$. Consider a sequence $\{s_n\}$ defined by
$$s_1=0, s_2=a, s_n=\begin{cases}s_{n-1}+a,& 2^{2^{2k-2}}\le n<2^{2^{2k-1}}\text{ for
some } k\in\mathbb{N};\\
s_{n-1}+b,& 2^{2^{2k-1}}\le n<2^{2^{2k}}\text{ for
some } k\in\mathbb{N}.
\end{cases}$$
Then 
$$\limsup_{n\to\infty}\frac1n s_n=a\text{ and }\liminf_{n\to\infty}\frac1n
s_n=b.$$

Let $t_n=\exp(-s_n)$, then
$$X=\{0,1\}\cup\{t_n:n\in\mathbb{N}\}$$
is a compact metric space with the induced metric from $\mathbb{R}$. Define
$$f(x)=\begin{cases}x, &x=0\text{ or }x=1;\\
t_{n-1},&x=t_n, n\in\mathbb{N}.
\end{cases}$$
Then $f$ is continuous. It is not difficult to see that
$h_L^+(f)=l(f)=a$ but $h_L^-(f)=b$.
We can even incorporate the idea of Example \ref{diffl} and obtain examples
of homeomorphisms
for which the strict inequality
$h_L^-(f)<h_L^+(f)<l(f)$ holds.
\end{example}

\begin{example}\label{exlbd}
This is the last example and it shows how the inequalities in Proposition
\ref{lbd} may be strict when $X_\infty$ is not a single point. Let $X=[0,1]\times
[0,1]/\sim$ be the cylinder parametrized
by $[0,1)\times [0,1]$, where $(0,y)\sim (1,y)$ for every $y\in[0,1]$.
Define
$$f(x,y)=\begin{cases} (2x,y), \text{if }2x<1;\\
(0,y), \text{if }2x\ge 1.
\end{cases}$$
Then $X_\infty=\{0\}\times [0,1]$.
For every $z_1\ne z_2\in\{0\}\times [0,1]$, $D(f^{-n}(z_1),f^{-n}(z_2))\ge
 d(z_1,z_2)$, hence 
 $$\sup_{x\ne y\in X_\infty}\limsup_{n\to\infty}\frac1n\log\frac{d(x,y)}{D(f^{-n}(x),f^{-n}(y))}=0.$$
But for every open cover $\calu$ of diameter sufficiently small, consider
the radius of the ball centered at $(0,0)$ that can be covered by $f^{-n}(\calu)$, we
see
$h_L^*(f,\calu)=\log
2$.
\end{example}

\section*{Acknowledgments} 
This work is partially
supported by the Center of Dynamics and Geometry in Department of Mathematics,
Pennsylvania State University.  I would like to thank Anatole Katok and Michal Misiurewicz for
reading the drafts of the paper and providing numerous helpful
comments.

\medskip


\begin{thebibliography}{99}


\bibitem{Bowen}(MR0338317)
\newblock R. Bowen
\newblock \emph{Topological entropy for noncompact sets},
\newblock Trans. Amer. Math. Soc. {\bfseries 184} (1973), 125--136.


\bibitem{DZG}(MR1667544)
\newblock X. Dai, Z. Zhou and X. Geng,
\newblock \emph{Some relations between Hausdorff-dimensions and entropies},
\newblock Sci. China Ser. A {\bfseries 41} (1998), 1068--1075.


\bibitem{KB}(MR1326374)
\newblock A. Katok and B. Hasselblatt,
\newblock ``Introduction to the Modern Theory of Dynamical Systems,"
\newblock Cambridge University Press, 1995.


\bibitem{Misiu}(MR2018882)
\newblock M. Misiurewicz,
\newblock \emph{On Bowen's definition of topological entropy},
\newblock Discrete Contin. Dyn. Syst. {\bfseries 10} (2004), 827--833. 

\bibitem{PW}(MR0648108)
\newblock P. Walters,
\newblock ``An Introduction to Ergodic Theory,"
\newblock Springer-Verlag, 1982.


\end{thebibliography}
\end{document}